\documentclass[11pt]{article}
\usepackage[table]{xcolor}
\usepackage{graphicx} 
\usepackage{comment}
\usepackage{amsmath,amsfonts,amssymb}
\usepackage{graphicx} \usepackage{multirow} \usepackage{tikz,pgf}
\usepackage{url} \usepackage{amsmath} \usepackage{graphicx}
\usepackage{epsfig} \usepackage{epstopdf} \usepackage{color}
\usepackage{enumitem} \usepackage{ amsfonts} \usepackage{amsmath}
\usepackage{tikz}
\usepackage{pgfmath}
\usepackage{tikz-3dplot}
\usetikzlibrary{math}
\usepackage{amsthm} \parindent 0em \parskip 0.5em
\usepackage{mdframed}
\newmdenv[linecolor=black,skipabove=\topsep,skipbelow=\topsep,
leftmargin=-3pt,rightmargin=-3pt,
innerleftmargin=6pt,innerrightmargin=6pt]{mybox}
 \def\tto{\;{\lower 1pt
		\hbox{$\rightarrow$}}\kern -10pt \hbox{\raise 2pt
		\hbox{$\rightarrow$}}\;} \def\Hat{\widehat} 
\def\Bar{\overline} \def\ra{\rangle} \def\la{\langle}
 \def\epsilon{\varepsilon} \def\B{\Bbb B}
\def\h{\hfill\Box} \def\R{\Bbb R} 
 \def\N{\Bbb N}

 \def\gph{\mbox{\rm gph}}
 
 \def\dim{\mbox{\rm dim}}
\def\dom{\mbox{\rm dom}}

 \def\h{\hfill\square} 
\def\ph{\varphi}  
\def\oR{\Bar{\R}}

\def\ph{\varphi}  
\def\oR{\Bar{\R}}

\setlist[enumerate,1]{itemsep=0.0ex,parsep=0.5ex,label={\rm(\alph*)},leftmargin=*,
	align=left} \newcounter{lk}

\usepackage[colorlinks=true]{hyperref}

\begin{document} 
	\begin{center}
		{\sc\bf On Strong Quasiconvexity of Functions in Infinite Dimensions}\\[1ex]
		 {\sc Nguyen Mau   Nam}\footnote{Fariborz Maseeh Department of Mathematics and Statistics,
			Portland State University, Portland, OR 97207, USA (email: mnn3@pdx.edu). },
			{\sc Jacob Sharkansky\footnote{Fariborz Maseeh Department of Mathematics and Statistics,Portland State University, Portland, OR 97207, USA (emails: jacshark@pdx.edu, jacobsh49@gmail.com).}}
			\end{center}
	\small{\bf Abstract.} In this paper, we explore the concept of $\sigma$-quasiconvexity for functions defined on normed vector spaces. This notion encompasses two important and well-established concepts: quasiconvexity and strong quasiconvexity. We start by analyzing certain operations on functions that preserve $\sigma$-quasiconvexity. Next, we present new results concerning the strong quasiconvexity of norm and Minkowski functions in infinite dimensions. Furthermore, we extend a recent result by F. Lara \cite{Lara22} on the supercoercive properties of strongly quasiconvex functions, with applications to the existence and uniqueness of minima, from finite dimensions to infinite dimensions. Finally, we address counterexamples related to strong quasiconvexity.
	\\[1ex]

	\newtheorem{theorem}{Theorem}[section]
	\newtheorem{proposition}[theorem]{Proposition}
 \newtheorem{lemma}[theorem]{Lemma}
	\newtheorem{corollary}[theorem]{Corollary}
	\theoremstyle{definition}
	\newtheorem{remark}[theorem]{Remark}
	\newtheorem{definition}[theorem]{Definition}
	\newtheorem{example}[theorem]{Example}
	
	\normalsize
	\section{Introduction and Preliminaries}

The concept of convexity for functions plays a significant role in numerous applications. Due to its critical importance in convex analysis and optimization, and its wide-ranging applications across various fields, this notion has been extensively studied since the early 20th century. Convex functions possess several favorable general properties, making them particularly useful in both theoretical and numerical optimization aspects. For basic definitions, theories, and applications of convex functions, see \cite{bc,Bertsekas2003,bl,Boyd2004,Urruty1993,Nesterov2005,MN22,Rockafellar1970,Zalinescu2002} and references therein.

The significant roles that convex functions play in various applications have inspired the search for new classes of functions beyond convexity, which retain the desirable properties of convex functions while offering broader applications. This effort has led to the development of numerous concepts of {\em generalized convexity} for functions, with {\em quasiconvexity} and {\em strong quasiconvexity} being two important notions.
\begin{definition} Let $f\colon \Omega\to \oR=[-\infty, \infty]$ be a function defined on a nonempty convex set $\Omega$ in  normed vector space $X$, and let $\sigma\geq 0$. The function $f$ is said to be $\sigma$-{\em quasiconvex} if 
\begin{equation*}
    f(\lambda x+(1-\lambda)y)\leq \max\{f(x), f(y)\}-\frac{\sigma}{2}\lambda(1-\lambda)\|x-y\|^2\; \mbox{\rm whenever }x, y\in \Omega \; \mbox{\rm and }0<\lambda<1.
\end{equation*}
\begin{enumerate}
    \item If $\sigma=0$, then we say that $f$ is {\em quasiconvex};
    \item If $\sigma>0$, then we say that $f$ is {\em strongly quasiconvex} with parameter $\sigma$, or $f$ is $\sigma$-{\em strongly quasiconvex}.
    \end{enumerate}
\end{definition}
The readers are referred to \cite{QS2005, QS1997, QS2014} and the references therein for a well-developed theory of quasiconvexity. Historical remarks on quasiconvexity and its applications to economics can be found in \cite{QS2004}. The concept of strong quasiconvexity was introduced in \cite{SQC1966}. This notion has been strongly developed in both theories and numerical algorithms by F. Lara and others; see, e.g., \cite{Lara22,Lara23,Lara241,Lara242}.

Since most recent works on strong quasiconvexity focus on finite-dimensional settings, the aim of this paper is to explore this concept in general normed vector spaces, with particular emphasis on uniformly convex Banach spaces. In particular, we extend several results from \cite{Lara22}. These new developments provide a foundation for advancing both theoretical studies and numerical algorithms involving strongly quasiconvex functions beyond finite-dimensional spaces and even Hilbert spaces.
Our paper is organized as follows. In Section 2, we examine several operations on functions that preserve $\sigma$-quasiconvexity. Section 3 focuses on the strong quasiconvexity of norms, extending known results from \cite{J} regarding the strong quasiconvexity of the Euclidean norm in $\R^n$. In Section 4, we generalize \cite[Theorem~1]{Lara22}, concerning the supercoercive properties of strongly quasiconvex functions, to the infinite-dimensional setting. Building on this, Section 5 extends certain existence results from \cite{Lara22}, applying them to demonstrate the proper properties of the proximal mapping for strongly quasiconvex functions in reflexive Banach spaces. Finally, in Section 6, we present some counterexamples related to strong quasiconvexity.

Throughout the paper, we consider {\em nontrivial} normed vector spaces over $\R$ unless otherwise state. The open (resp., closed) ball in a normed vector space $X$ with center $a$ and radius $r>  0$ is denoted by $\B(a; r)$ (resp., $\Bar{\B}(a; r)$). We use the convention that $\inf(\emptyset)=\infty$.

\section{Operations Preserving $\sigma$-Quasiconvexity}

In this section, we examine several operations on functions that preserve $\sigma$-quasiconvexity. While this property is not generally preserved under summation, it exhibits some desirable properties under composition, maximization, infimum, and supremum. 

We begin with an elementary result regarding the preservation of $\sigma$-quasiconvexity under scalar multiplication.
\begin{proposition} Let $f\colon \Omega\to \oR$ be a convex function defined on a nonempty convex subset $\Omega$ of a normed vector space, and let $c> 0$. If $f$ is $\sigma$-quasiconvex, then $cf$ is $c\sigma$-quasiconvex.    
\end{proposition}
Next, we discuss a result on $\sigma$-quasiconvexity under composition. 
\begin{proposition} Let $f\colon \Omega\to \R$ be a real-valued convex function defined on a nonempty convex subset $\Omega$ of a normed space $X$, and let $\ph\colon \R\to \oR$ be a function that is $\sigma$-quasiconvex and nondecreasing on an interval $I$ containing $f(\Omega)$. Suppose further that there exists $\ell\geq 0$ such that
\begin{equation*}
    \ell \|x-u\|\leq |f(x)-f(u)|\; \mbox{\rm for all }x, u\in \Omega.
\end{equation*}
Then $\ph\circ f$ is $\sigma \ell^2$-quasiconvex on $\Omega$.
\end{proposition}
\begin{proof}
Take any $x, u \in \Omega$ and $0<\lambda<1$. By the convexity of $f$ we have
\begin{equation*}
    f(\lambda x+(1-\lambda)u)\leq \lambda f(x)+(1-\lambda)f(u).
\end{equation*}
Then
\begin{equation*}
\begin{aligned}
    \ph(f(\lambda x+(1-\lambda)u))&\leq \ph(\lambda f(x)+(1-\lambda)f(u))\\
    &\leq \max\{\ph(f(x)), \ph(f(u))\}-\frac{\sigma}{2}\lambda (1-\lambda)(f(x)-f(u))^2\\
    &\leq \max\{\ph(f(x)), \ph(f(u))\} -\frac{\sigma}{2}\ell^2\|x-u\|^2.
    \end{aligned}
\end{equation*}
Therefore, $\ph\circ f$ is $\sigma\ell^2$-quasiconvex.
\end{proof}
\begin{proposition}
Let \(\Omega\) be a nonempty convex subset of  a normed vector space,  and let \(\{f_\alpha\}_{\alpha \in I}\) be a nonempty collection of $\sigma_\alpha$-quasiconvex functions from \(\Omega\) to $\oR$, where $\sigma_\alpha\geq 0$ for all $\alpha\in I$. Let 
\begin{equation*}
    \sigma = \inf_{\alpha \in I} \sigma_\alpha.
\end{equation*}
Then \(f = \sup_{\alpha \in I} f_\alpha\) is \(\sigma\)-quasiconvex. In particular, $f$ is strongly quasiconvex if $\sigma>0$.
\end{proposition}
\begin{proof}
By assumption, for all \(\alpha \in I\), \(x, y \in \Omega\), and \(\lambda \in (0, 1)\), we have
\begin{equation*}
\begin{aligned}
    f_\alpha(\lambda x + (1 - \lambda)y) &\leq \max\{f_\alpha(x), f_\alpha(y)\} - \frac{\sigma_\alpha}{2} \lambda(1 - \lambda)\|x - y\|^2 \\
    &\leq \max\{f_\alpha(x), f_\alpha(y)\} - \frac{\sigma}{2} \lambda(1 - \lambda)\|x - y\|^2.
\end{aligned}
\end{equation*}
Take the supremum over all \(\alpha \in I\) to get
\begin{equation*}
    f(\lambda x + (1 - \lambda)y) \leq \sup_{\alpha \in I}(\max\{f_\alpha(x), f_\alpha(y)\}) - \frac{\sigma}{2} \lambda(1 - \lambda)\|x - y\|^2.
\end{equation*}
Note that
\begin{equation*}
    \sup_{\alpha \in I}(\max\{f_\alpha(x), f_\alpha(y)\}) \leq \max\{f(x), f(y)\}.
\end{equation*}
Therefore,
\begin{equation*}
    f(\lambda x + (1 - \lambda)y) \leq \max\{f(x), f(y)\} - \frac{\sigma}{2} \lambda(1 - \lambda)\|x - y\|^2.
\end{equation*}
We conclude that \(f\) is \(\sigma\)-quasiconvex.
\end{proof}

\begin{corollary} Let $\Omega$ be a nonempty convex subset of a normed vector space, and let $f_i\colon \Omega\to \oR$ for $i=1, \ldots, m$ be $\sigma_i$-quasiconvex, where $\sigma_i\geq 0$ for all $i=1, \ldots, m$. Then $\max_{i=1, \ldots, m}f_i$ is $\sigma$-quasiconvex, where $\sigma=\min_{i=1, \ldots, m}  \sigma_i$. In particular, if $f_i$ is strongly quasiconvex for all $i=1, \ldots, m$, then $\max_{i=1, \ldots, m}f_i$ is strongly quasiconvex. 
    
\end{corollary}

Let $X, Y$ be normed vector spaces, let $\Omega$ be a nonempty subset of $X$,  and let $\Theta$ be a nonempty  subset of $Y$. Consider a set-valued map $F\colon \Omega \tto \Theta$ and define the {\em graph} of $F$ by
\begin{equation*}
    \gph(F)=\big\{(x, y)\in \Omega\times \Theta\; \big |\; y\in F(x)\big\}\subset X\times Y.
\end{equation*}
We say that $F$ is a {\em convex set-valued map} if its graph is a convex set in $X\times Y$. The {\em domain} of $F$ is defined by
\begin{equation*}
    \dom(F)=\big\{x\in \Omega\; \big |\; F(x)\neq\emptyset\big\}.
\end{equation*}
Given a function $f\colon  \Theta\to \oR$, the {\em optimal value function} associated with $F$ and $f$ is defined by
	\begin{equation}\label{key1}
		\mu(x)=\inf \{f(y) \mid y \in F(x)\},\ \;  x \in \Omega.	
	\end{equation}
Thus, $\mu\colon \Omega\to \oR$ is an extended-real-valued function.

\begin{definition}Let $X, Y$ be normed vector spaces, let $\Omega$ be a nonempty convex subset of $X$,  and let $\Theta$ be a nonempty  subset of $Y$. Consider a set-valued map $F\colon \Omega\tto \Theta$.
\begin{enumerate}
    \item The set-valued map $F$ is said to be $\ell$-{\em Lipschitz}, where $\ell> 0$, if 
    \begin{equation*}
        \|y_1-y_2\|\leq \ell \|x_1-x_2\|\; \mbox{\rm whenver }y_1\in F(x_1), \; y_2\in F(x_2).
    \end{equation*}
    \item The set valued map $F$ is said to be $\gamma$-{\em expansive}, where $\gamma>0$, if  
    \begin{equation*}
         \|x_1-x_2\|\leq \gamma\|y_1-y_2\|\; \mbox{\rm whenver }y_1\in F(x_1), \; y_2\in F(x_2).
    \end{equation*}
\end{enumerate}
 \end{definition} 
Note that if $F\colon \Omega \tto \Theta$ is $1$-Lipschitz, then it is called {\em nonexpansive}. If $F$ is $\ell$-Lipschitz, then $F(x)$ must be a single-point set for every $x\in \dom(F)$; see \cite[Definition 6.1.3]{AT03} and the discussion therein. Furthermore, $F\colon \Omega \tto \Theta$ is $\gamma$-expansive if and only if $F^{-1}\colon \Theta\tto \Omega$ is $\gamma$-Lipshitz.

\begin{theorem}\label{marginal QS} Consider the optimal value function \eqref{key1}, where $\Omega$ and $\Theta$ are convex. Suppose that $f\colon \Theta\to \oR$ is $\sigma$-quasiconvex and $F\colon \Omega \tto \Theta$ is a convex set-valued map which is $\gamma$-expansive.  
Then $\mu$ is $\frac{\sigma}{\gamma^2}$-quasiconvex.
\end{theorem}
\begin{proof}
Take $x_1, x_2 \in \Omega$ such that $\mu(x_i)\in [-\infty, \infty)$ for $i=1, 2$, and let $\lambda \in(0,1)$. Then there exist sequences $y_{1,k} \in F(x_1)$ and  $y_{2,k} \in F(x_2)$ with
	$$
	f(y_{i,k})\to \mu(x_i) \text { for } i=1,2 .
	$$
Fix any $k\in \N$. Then, it follows from the quasiconvexity of $f$ and the expansive property of $F$ that
	\begin{equation}\label{key2}
		\begin{aligned}
		f(\lambda y_{1,k}+(1-\lambda)y_{2,k}) & \leq \max\{f(y_{1,k}),f(y_{2,k})\}-\frac{\sigma}{2} \lambda(1-\lambda)\|y_{1,k}-y_{2,k}\|^2\\
  &\leq \max\{f(y_{1,k}),f(y_{2,k})\}-\frac{\sigma}{2 \gamma^2}\lambda(1-\lambda) \|x_1-x_2\|^2.
		\end{aligned}
	\end{equation}
	Furthermore, since $\gph(F)$ is convex, we have
	$$
	\left(\lambda x_1+(1-\lambda) x_2, \lambda y_{1,k}+(1-\lambda) y_{2,k}\right)=\lambda\left(x_1, y_{1,k}\right)+(1-\lambda)\left(x_2, y_{2,k}\right) \in \operatorname{gph} F,
	$$
	and therefore $\lambda y_{1,k}+(1-\lambda) y_{2,k} \in F\left(\lambda x_1+(1-\lambda) x_2\right)$. Using the given assumption together with~\eqref{key2} implies that
\begin{equation*}
\begin{aligned}
	\mu(\lambda x_1+(1-\lambda) x_2) & \leq \max\{f(y_{1,k}),f(y_{2,k})\}-\frac{\sigma}{2 \gamma^2}\lambda(1-\lambda) \|x_1-x_2\|^2.
\end{aligned}
\end{equation*}
Letting finally $k\to\infty$ ensures that
\begin{equation*}
\mu(\lambda x_1+(1-\lambda) x_2)  \leq \max\{\mu(x_1), \mu(x_2)\}-\frac{\sigma}{2 \gamma^2}\lambda(1-\lambda) \|x_1-x_2\|^2.
\end{equation*}
Since this inequality holds obviously if $\mu(x_1)=\infty$ or $\mu(x_2)=\infty$, we see that $\mu$ is $\frac{\sigma}{\gamma^2}$-quasiconvex. 
\end{proof}

 \begin{corollary} Let $f\colon Y\to \oR$ be a $\sigma$-quasiconvex function defined on a normed vector space $X$,  let $A\colon Y\to X$ be a nonzero continuous linear mapping between normed vector spaces, and let $b\in X$. Define $B(z)=A(z)+b$ for $z\in Y$, and let $\Omega$ be a nonempty convex subset of $B(Y)$. Define
 \begin{equation*}
     \mu(x)=\inf\{f(y)\; |\; y\in B^{-1}(\{x\})\}, \ \; x\in \Omega.
 \end{equation*}
  Then $\mu$ is $\frac{\sigma}{\|A\|^2}$-quasiconvex.   
 \end{corollary}
 \begin{proof} We only need to use Theorem \ref{marginal QS} for $F(x)=B^{-1}(\{x\})$ for $x\in X$. 
 \end{proof}

 \begin{corollary}\label{a comp}Let $X$ and $Y$ be normed vector spaces,  and let $B\colon X\to Y$ be an affine mapping given by $B(x)=A(x)+b$ for all $x\in X$, where $b\in Y$. Suppose that there exists a constant $\gamma>0$ such that
 \begin{equation*}
      \|x\|\leq \gamma\|Ax\|\; \mbox{\rm for all }x\in X.
 \end{equation*}
If $f$ is $\sigma$-quasiconvex on $\Theta:=B(\Omega)$, where $\Omega$ is a nonempty convex set in $X$,  
then the composition $f\circ B$ is $\frac{\sigma}{\gamma^2}$-quasiconvex on $\Omega$. 
 \end{corollary}
\begin{proof}
We only need to use Theorem \ref{marginal QS} for $F(x)=\{Bx\}$.
\end{proof}

\begin{comment}
 \begin{corollary} Let $f\colon \Theta\to \R$ be a $\sigma$-strongly quasiconvex function. Then for any $a\in X$ and $b\in \R$, the funcion
 \begin{equation*}
     \ph(x)=f(x+a)+b,\;  x\in X,
 \end{equation*}
is also $\sigma$-strongly quasiconvex on $\Theta-a$. 
 \end{corollary}
 {\bf Proof}. Using the previous result with $B(x)=x+a$ for $x\in \Omega:=\Theta-a$ yields the $\sigma$-strong quasiconvexity of $f(\cdot+a)$ on $\Omega-a$. Then $\ph$ is $\sigma$-strongly quasiconvex as the sum of a $\sigma$-strongly quasiconvex function and a constant. $\h$

 \begin{definition}

 \end{definition}

\end{comment}

 \begin{corollary} Let $f_1, f_2\colon X\to (-\infty, \infty]$, where $X$ is a normed vector space. Suppose that the function $f\colon X\times X\to (-\infty,\infty]$ defined by
 \begin{equation*}
     f(x_1, x_2)=f_1(x_1)+f_2(x_2), \ \; (x_1, x_2)\in X\times X,
 \end{equation*}
 is $\sigma$-quasiconvex. 
 Then the infimal convolution
 \begin{equation*}
     (f_1\square f_2)(x)=\inf\{f_1(x_1)+f_2(x_2)\; |\; x_1+x_2=x\}, \; x\in X,
 \end{equation*}
is $\sigma$-quasiconvex.  Here we use the ``sum'' norm on $X\times X$, defined by \(\|(x_1, x_2)\| = \|x_1\| + \|x_2\|\) for all \((x_1, x_2) \in X \times X\).
 \end{corollary}
 \begin{proof}
 Define $F\colon X\tto X\times X$ by
 \begin{equation*}
     F(x)=\{(x_1, x_2)\in X \times X\; |\; x_1+x_2=x\},\ \;  x\in X.
 \end{equation*}
Note that $F$ is convex. Take $(x_1, x_2)\in F(x)$ and $(u_1, u_2)\in F(u)$. Then
 \begin{equation*}
     \|(x_1, x_2)-(u_1, u_2)\|=\|x_1-u_1\|+\|x_2-u_2\|\geq \|x_1+x_2-(u_1+u_2)\|=\|x-u\|.
 \end{equation*}
Thus, $F$ is also $1$-expansive. By Theorem \ref{marginal QS}, the function $f_1\square f_2$ is $\sigma$-quasiconvex.  
\end{proof}

Let us now extend the result of \cite[Theorem~1]{J} to infinite dimensions. We provide a detailed proof for the convenience of the readers. 
 \begin{proposition}\label{DSQC} Let $X$ be a normed vector space, let $\Omega$ be a nonempty convex subset of $\Omega$, and let $f\colon \Omega\to \oR$ be a $\sigma$-quasiconvex function. Then for any $x, u\in \Omega$ such that $x\neq u$, the function
 \begin{equation*}
     \ph(t)=f(tx +(1-t)u)
 \end{equation*}
is $\sigma\|x-u\|^2$-quasiconvex on $[0, 1]$. Furthermore, $f$ is $\sigma$-quasiconvex on $\Omega$ iff for any $x, u\in \Omega$ such that $x\neq u$, the function
 \begin{equation*}
     \ph(t)=f\big(u+\frac{t}{\|x-u\|}(x-u)\big), \; t \in \R,
 \end{equation*}
    is $\sigma$-quasiconvex on $[0, \|x-u\|]$.
 \end{proposition}
\begin{proof}
 Let $B\colon \R\to X$ be given by
 \begin{equation*}
     B(t)=u+t(x-u), \; t\in \R.
 \end{equation*}
 Define further $A(t)=t(x-u)$ for $t\in \R$. Then 
 \begin{equation*}
     \|x-u\|\, |t|\leq \|A(t)\|\; \mbox{\rm for all }t\in \R.
 \end{equation*}
 Thus, the first conclusion follows from the Corollary~\ref{a comp} because $f$ is $\sigma$-quasiconvex on $[x, u]=B([0, 1])\subset\Omega$. \\[1ex]
Now, consider the affine mapping
\begin{equation*}
    \Hat{B}(t)=u+t\frac{x-u}{\|x-u\|},\; \ t\in \R,
\end{equation*}
and define further the linear function $\Hat{A}(t)=t\frac{x-u}{\|x-u\|}$ for $t\in \R$. We can easily see that $\|\Hat{A}(t)\|\geq |t|$. Then the implication $\Longrightarrow$ in the second conclusion also follows from Corollary~\ref{a comp} because $\Hat{B}([0, \|x-u\|)=[x, u]\subset\Omega$. Now, let us prove the converse implication.

Take any $x, u\in X$ such that $x\neq u$ and $0<t<1$. Then
\begin{equation*}
    f(tx+(1-t)u)=f(u+t(x-u))=f(u+\frac{t\|x-u\|}{\|x-u\|}(x-u))=\ph(t \|x-u\|).
\end{equation*}
By the strong quasiconvexity of $\ph$ on $[0, \|u-v\|]$, we have
\begin{equation*}
\begin{aligned}
    \ph(t\|x-u\|)&=\ph(t\|x-u\|+(1-t)0)\\
    &\leq t\ph(\|x-u\|)+(1-t)\ph(0)-\frac{\sigma}{2} t(1-t)\|x-u\|^2\\
    &=tf(x)+(1-t)f(u)-\frac{\sigma}{2} t(1-t)\|x-u\|^2.
    \end{aligned}
\end{equation*}
Therefore, $f$ is $\sigma$-strongly convex.
\end{proof}

\begin{proposition} \label{marginal min}
    Let $\Omega$ and $\Theta$ be two nonempty convex sets in normed vector spaces $X$ and $Y$, respectively. Assume that \( f\colon \Omega\times \Theta\to \oR \) is  $\sigma$-quasiconvex, and \( F\colon  \Omega\tto \Theta \) is a convex set-valued map. Then the optimal value function 
\[
v(x) = \inf_{y \in F(x)} f(x, y)
\]
is $\sigma$-quasiconvex on $\Omega$. Here we use the \(p\)-norm on $X\times Y$ for any \(1 \leq p \leq \infty\), defined by
\begin{equation*}
\begin{aligned}
    \|(x, y)\| &= \left(\|x\|^p + \|y\|^p\right)^{1 / p} \hspace{8 pt} \text{for \(p < \infty\)}, \\
    \|(x, y)\| &= \max\{\|x\|, \|y\|\} \hspace{8 pt} \text{for \(p = \infty\)}.
\end{aligned}
\end{equation*}
\end{proposition}
\begin{proof}
Note that regardless of which norm on \(X \times Y\) we use, for all \((x, y) \in X \times Y\), we have
\begin{equation*}
    \|x\| \leq \|(x, y)\|.
\end{equation*}
Fix  each  $x_i \in \Omega$ for $i=1, 2$ and $0<\lambda<1$. Let $y_{i,k} \in F(x_i)$ for $i=1, 2$  be such that
	\[
	  f(x_i, y_{i,k})\to  v(x_i) .
	\]

	Now, consider the point $x_\lambda = \lambda x_1 + (1 - \lambda) x_2 $ and fix any $k\in \N$. Then we have
	\[
	v(x_\lambda) \leq f(x_\lambda, \lambda y_{1,k} + (1 - \lambda) y_{2,k}).
	\]
	since  $\lambda y_{1,k} + (1 - \lambda) y_{2,k} \in F(x_\lambda) $ by the convexity of $F$.
	By the $\sigma$-quasiconvexity of \( f \), we know
	\[
	f(x_\lambda, \lambda y_{1,k} + (1 - \lambda) y_{2,k}) \leq \max\{f(x_1, y_{1,k}), f(x_2, y_{2,k})\} - \lambda(1 - \lambda) \frac{\sigma}{2}  \|(x_1 - x_2, y_{1,k} - y_{2,k})\|^2.
	\]
	Then  we get
	\[
        v(x_\lambda) < \max\{f(x_1, y_{1,k}), f(x_2, y_{2,k})\} - \lambda(1 - \lambda) \frac{\sigma}{2}  \|(x_1 - x_2, y_{1,k} - y_{2,k})\|^2.
        \]
        Therefore, we get
	\[
	v(x_\lambda) < \max\{f(x_1, y_{1,k}), f(x_2, y_{2,k})\} - \lambda(1 - \lambda) \frac{\sigma}{2} \|x_1 - x_2\|^2.
	\]
	
	Finally, letting $k\to \infty$, we conclude
	\[
	v(x_\lambda) \leq \max\{v(x_1), v(x_2)\} - \lambda(1 - \lambda) \frac{\sigma}{2} \|x_1 - x_2\|^2.
	\]
	Thus,  $v$ is $\sigma$-quasiconvex.
\end{proof}

\begin{definition} Let $X, Y$ be normed vector spaces, let $\Omega$ be a nonempty convex subset of $X$,  and let $\Theta$ be a nonempty  subset of $Y$.  A set-valued map $F\colon \Omega\tto\Theta$ is said to be an {\em affine process} if
\begin{equation*}
    F(\lambda x+(1-\lambda)y)\subset  \lambda F(x)+(1-\lambda)F(y)\;\ \mbox{\rm for all }\; x, y\in \Theta \; \mbox{\rm and }0<\lambda<1.
\end{equation*}
\end{definition}

\begin{proposition}  Let $\Omega$ and $\Theta$ be two nonempty convex sets in normed vector spaces $X$ and $Y$, respectively. Assume that \( f\colon \Omega\times \Theta\to \oR \), and \( F\colon  \Omega\tto \Theta \) is a set-valued map. Define the function
\begin{equation}\label{Vmax}
    V(x)=\sup\big\{f(x, y)\; \big|\; y\in F(x)\big\}, \ \; x\in \Omega.
\end{equation}
Suppose that $f$ is $\sigma$-quasiconvex and that $F$ is an affine process such that $\dom(F)=\Omega$. Then the function $V\colon \Omega\to \oR$ defined in \eqref{Vmax} is $\sigma$-quasiconvex. Here we use the same $p$-norm on $X\times Y$ for any $1\leq p\leq \infty$ as in Proposition \ref{marginal min}.
\end{proposition}
\begin{proof} Take any $x_1, x_2\in \Omega$ and $0<\lambda<1$. Fix any $\alpha\in \R$ such that $\alpha <V(\lambda x_1+(1-\lambda)x_2)$. By definition, there exists $y\in F(\lambda x_1+(1-\lambda)x_2)$ such that $\alpha<f(\lambda x_1+(1-\lambda)x_2, y)$.
Since $F$ is an affine process, we see that
\begin{equation*}
    y\in F(\lambda x_1+(1-\lambda)x_2)\subset \lambda F(x_1)+(1-\lambda)F(x_2).
\end{equation*}
Then $y=\lambda y_1+(1-\lambda)y_2$ for some $y_1\in F(x_1)$ and $y_2\in F(x_2)$. It follows that
\begin{equation*}
\begin{aligned}
    \alpha&=f(\lambda x_1+(1-\lambda)x_2, y)\\
    &=f(\lambda(x_1, y_2)+(1-\lambda)(x_2, y_2)) \\
    &\leq \max\{f(x_1, y_1), f(x_2, y_2)\}-\frac{\sigma}{2}\lambda(1-\lambda)\|(x_1, y_1)-(x_2, y_2)\|^2\\
    &\leq \max\{V(x_1), V(x_2)\}-\frac{\sigma}{2}\lambda(1-\lambda)\|x_1-x_2\|^2.
    \end{aligned}
\end{equation*}
This implies that
\begin{equation*}
    V(\lambda x_1+(1-\lambda)x_2)\leq \max\{V(x_1), V(x_2)\}-\frac{\sigma}{2}\lambda(1-\lambda)\|x_1-x_2\|^2.
\end{equation*}
Therefore, $V$ is $\sigma$-quasiconvex. 
\end{proof}

\section{Strong Quasiconvexity of Norms}

 Let us first provide an obvious extension \cite[Theorem~2]{J} to inner product spaces. We include its detailed proof for the convenience of the readers. 
\begin{proposition}\label{HilbertSQCnorm} Let $H$ be a real inner product space. Then the normed function $f(x)=\|x\|$, $x\in H$, is strongly quasiconvex on any nonempty bounded convex subset $\Omega$ of $H$.     
 \end{proposition}
\begin{proof}
 Suppose that $\Omega\subset \B(0; r)$. Take any $x, u\in \Omega$ with $x\neq u$. It suffices to show that the function
 \begin{equation*}
     \ph(t)=f(u+\frac{t}{\|x-u\|}(x-u))
 \end{equation*}
is \(\sigma\)-strongly quasiconvex on $[0, \|x-u\|]$, where \(\sigma > 0\) is independent of \(x\) and \(u\). We then have
\begin{equation*}
\begin{aligned}
    \ph(t)&=\sqrt{\Big\|u+\frac{t}{\|x-u\|}(x-u)\Big\|^2}=\sqrt{t^2+2\frac{t}{\|x-u\|}\la u, x-u\ra+\|u\|^2}\\
    &=\sqrt{\Big(t+\frac{\la u, x-u\ra}{\|x-u\|}\Big)^2+\|u\|^2-\frac{\la u, x-u\ra^2}{\|x-u\|^2}}.
    \end{aligned}
\end{equation*}
Let \(a = \frac{\la u, x-u\ra}{\|x-u\|}\) and \(b = \|u\|^2\), then
\begin{equation*}
    \ph(t) = \sqrt{(t + a)^2 + b^2 - a}.
\end{equation*}
Note that \(b^2 - a \geq 0\) by the Cauchy-Schwartz inequality. If \(b^2 - a > 0\), then by \cite[Example 1]{J}, \(\ph\) is \(\sigma^1_{u,x}\)-strongly quasiconvex on \([0, \|x - u\|]\), where \(\sigma^1_{u,x}\) is
\begin{equation*}
    \sigma^1_{u,x} = \frac{1}{\max\{\ph(0), \ph(\|x - u\|)\}} = \frac{1}{\max\{\|u\|, \|x\|\}} \geq \frac{1}{r}.
\end{equation*}
If \(b^2 - a = 0\), then \(\ph(t) = |t + a|\) and by \cite[Example 2]{J}, \(\ph\) is \(\sigma^2_{u,x}\)-strongly quasiconvex on \([0, \|x - u\|]\), where \(\sigma^2_{u,x}\) is
\begin{equation*}
    \sigma^2_{u,x} = \frac{2}{\|u - x\|} \geq \frac{1}{r}.
\end{equation*}
Set \(\sigma = \frac{1}{r}\), then \(\ph\) is \(\sigma\)-strongly quasiconvex on \([0, \|x - u\|]\). By Proposition~\ref{DSQC}, \(f\) is \(\sigma\)-strongly quasiconvex on \(\Omega\).
\end{proof}
The remainder of this section is devoted to presenting a necessary and sufficient condition on the norm of a Banach space for the strong quasiconvexity of the norm function on bounded convex sets. For this purpose we distinguish between the cases \(\dim(X) \geq 2\) and the trivial case \(\dim(X) = 1\). If \(\dim(X) = 1\) then we may assume that \(X\) is \(\mathbb{R}\). A straightforward argument shows that there must be a positive scalar multiple \(c\) such that \(\|x\| = c|x|\). Endowing \(X\) with the inner product \(\langle x, y \rangle = c^2xy\) yields \(\| \cdot \|\) as the corresponding norm. Therefore, \(X\) is a Hilbert space, at which point Proposition~\ref{HilbertSQCnorm} may be applied to show that the norm of \(X\) is strongly quasiconvex on bounded subsets.

For the rest of this section, assume that \(\dim(X) \geq 2\).
\begin{definition}
Let \(X\) be a Banach space. Define the modulus of convexity \(\delta\colon [0, 2] \rightarrow \mathbb{R}\) by
\begin{equation*}
    \delta(\epsilon) = \inf\left\{1 - \left\|\frac{x + y}{2}\right\|: \|x\| = \|y\| = 1, \|x - y\| \geq \epsilon\right\}.
\end{equation*}
\end{definition}
Equivalently, we may instead write
\begin{equation*}
    \delta(\epsilon) = \inf\left\{1 - \left\|\frac{x + y}{2}\right\|: \|x\|, \|y\| \leq 1, \|x - y\| \geq \epsilon\right\}.
\end{equation*}
For a detailed proof of this equivalence, see \cite[p. 60]{LT79}.

\begin{theorem}\label{NC}
Let \(X\) be a Banach space and \(\| \cdot \|\) be the corresponding norm. If for all nonempty bounded convex sets \(\Omega\) there exists \(\sigma > 0\) such that \(f(x) = \|x\|\) is \(\sigma\)-strongly quasiconvex on \(\Omega\), then \(\| \cdot \|\) is uniformly convex with modulus of convexity \(\delta(\epsilon) \geq \frac{\sigma_1\epsilon^2}{8}\), where \(\sigma_1 > 0\) is any value such that \(f\) is \(\sigma_1\)-strongly quasiconvex on the closed unit ball.
\end{theorem}
\begin{proof}
Let \(\sigma_1 > 0\) such that \(f\) is \(\sigma_1\)-strongly quasiconvex on the closed unit ball. Let \(0 < \epsilon \leq 2\) and \(x, y \in X\) such that \(\|x\| = \|y\| = 1\) and \(\|x - y\| \geq \epsilon\). Then
\begin{equation*}
    \left\|\frac{x + y}{2}\right\| \leq \max\{\|x\|, \|y\|\} - \frac{\sigma_1}{8}\|x - y\|^2 \leq 1 - \frac{\sigma_1}{8}\epsilon^2.
\end{equation*}
Set \(\delta = \frac{\sigma_1}{8}\epsilon^2\), then \(\left\|\frac{x + y}{2}\right\| \leq 1 - \delta\) so that \(\| \cdot \|\) is uniformly convex. In addition, we have
\begin{equation*}
    \frac{\sigma_1}{8}\epsilon^2 \leq 1 - \left\|\frac{x + y}{2}\right\|.
\end{equation*}
Therefore, \(\delta(\epsilon) \geq \frac{\sigma_1\epsilon^2}{8}\), which completes the proof. 
\end{proof}

Note that we can clearly weaken the assumption to be just that \(f\) is strongly quasiconvex on the closed unit ball, instead of every bounded convex set.

Next, we consider the following assumption:

{\bf Assumption (A)}:
\emph{\( X \) is a uniformly convex Banach space and the modulus of convexity \( \delta(\epsilon) \) satisfies the inequality}
\begin{equation*}
    \delta(\epsilon) \geq \frac{\sigma}{8}\epsilon^2 \quad \text{\em for all } \epsilon \in [0, 2].
\end{equation*}

To prove the main result of this section, we need two lemmas. The first lemma states that our assumptions guarantee that \(f\) is \(\sigma\)-strongly quasiconvex over the unit ball specifically for \(\lambda = \frac{1}{2}\).

\begin{lemma}\label{NormLemma1}
Under Assumption {\rm (A)}, for all \(x, y \in X\) such that \(\|x\|, \|y\| \leq 1\), we have
\begin{equation*}
    \left\|\frac{x + y}{2}\right\| \leq \max\{\|x\|, \|y\|\} - \frac{\sigma}{8}\|x - y\|^2.
\end{equation*}
\end{lemma}
\begin{proof} Indeed, first note that Assumption (A) implies that for all \(x, y \in X\) such that \(\|x\|, \|y\| \leq 1\), we have
\begin{equation*}
    \left\|\frac{x + y}{2}\right\| \leq 1 - \frac{\sigma}{8}\|x - y\|^2.
\end{equation*}
Suppose without loss of generality that \(x\neq 0\) or \(y\neq 0\) so that \(\max\{\|x\|, \|y\|\} > 0\). Let \(c = \max\{\|x\|, \|y\|\}\). Clearly, we have
\begin{equation*}
\begin{aligned}
    \left\|\frac{x}{c}\right\|\leq 1, \; \left\|\frac{y}{c}\right\| &\leq 1, \; \mbox{\rm and }    c \leq 1.
\end{aligned}
\end{equation*}
Therefore,
\begin{equation*}
\begin{aligned}
    \left\|\frac{x + y}{2}\right\| &=c\left\|\frac{x + y}{2c}\right\| \\
    &\leq c\left(1 - \frac{\sigma}{8}\left\|\frac{x - y}{c}\right\|^2\right) \\
    &= c - \frac{\sigma}{8c}\left\|x - y\right\|^2 \\
    &\leq \max\{\|x\|, \|y\|\} - \frac{\sigma}{8}\left\|x - y\right\|^2,
\end{aligned}
\end{equation*}
which completes the proof.
\end{proof}

The second lemma states that if \(f\) satisfies the inequality in the definition of \(\sigma\)-strong quasiconvexity on a set \(\Omega\) for \(\lambda = \frac{1}{2}\), where $\sigma>0$, then it is \(\frac{\sigma}{2}\)-strongly quasiconvex on this set.
\begin{lemma}\label{NormLemma2}
Let \(X\) be a normed vector space, \(\| \cdot \|\) be its corresponding norm, and let \(\Omega \subset X\) be a nonempty convex set. Suppose for any \(x, y \in \Omega\) we have
\begin{equation*}
    \left\|\frac{x + y}{2}\right\| \leq \max\{\|x\|, \|y\|\} - \frac{\sigma}{8}\|x - y\|^2.
\end{equation*}
Then \(f(x) = \|x\|\) is \(\frac{\sigma}{2}\)-strongly quasiconvex on \(\Omega\). That is, for any \(\lambda \in [0, 1]\) and \(x, y \in \Omega\), we have
\begin{equation*}
    \|\lambda x + (1 - \lambda)y\| \leq \max\{\|x\|, \|y\|\} - \frac{\sigma}{4}\lambda(1 -\lambda)\|x - y\|^2.
\end{equation*}
\end{lemma}
\begin{proof}
Fix \(x, y \in \Omega\) and let \(\lambda \in (0, 1)\). First assume that $\lambda \leq \frac{1}{2}$. Then we have
\begin{equation*}
\begin{aligned}
    \|\lambda x + (1 - \lambda)y\| &\leq \left\|\lambda x + \lambda y\right\| + (1 - 2\lambda)\|y\| \\
    &= 2\lambda\left\|\frac{x + y}{2}\right\| + (1 - 2\lambda)\|y\| \\
    &\leq 2\lambda \max\{\|x\|, \|y\|\} - 2\lambda \frac{\sigma}{8}\|x - y\|^2 + (1 - 2\lambda)\max\{\|x\|, \|y\|\} \\
    &= \max\{\|x\|, \|y\|\} - \frac{2\lambda}{\lambda(1 - \lambda)} \frac{\sigma}{8}\lambda(1 - \lambda)\|x - y\|^2 \\
    &= \max\{\|x\|, \|y\|\} - \frac{1}{1 - \lambda} \frac{\sigma}{4}\lambda(1 - \lambda)\|x - y\|^2 \\
    &\leq \max\{\|x\|, \|y\|\} - \frac{\sigma}{4}\lambda(1 - \lambda)\|x - y\|^2.
\end{aligned}
\end{equation*}
It follows that
\begin{equation*}
    \|\lambda x + (1 - \lambda)y\| \leq \max\{\|x\|, \|y\|\} - \frac{\sigma}{4}\lambda(1 -\lambda)\|x - y\|^2.
\end{equation*}
The above inequality similarly holds for \(\frac{1}{2} < \lambda < 1\) as well. The cases \(\lambda = 0\) and \(\lambda = 1\) are trivial, and so the proof is complete.
\end{proof}
We are now ready to state and prove the main result.
\begin{theorem}\label{StrongQuasiNorm}
Let \(X\) be a normed vector space, and let \(\| \cdot \|\) be its corresponding norm. Under Assumption {\rm (A)}, the norm function \(f(x) = \|x\|\) is a strongly quasiconvex function on every nonempty bounded convex set \(\Omega \subset X\). In particular, \(f\) is strongly quasiconvex with parameter \(\frac{\sigma}{2}\) on the closed unit ball.
\end{theorem}
\begin{proof} 
By Lemma~\ref{NormLemma1}, for any \(x, y \in X\) such that \(\|x\|, \|y\| \leq 1\), we have
\begin{equation*}
    \left\|\frac{x + y}{2}\right\| \leq \max\{\|x\|, \|y\|\} - \frac{\sigma}{8}\|x - y\|^2.
\end{equation*}
Using Lemma~\ref{NormLemma2}, we see that \(f\) is \(\frac{\sigma}{2}\)-strongly quasiconvex on the unit ball. Now, let \(\Omega \subset X\) be any nonempty bounded convex set. Then there exists \(c > 0\) such that \(\|x\| \leq c\) for all \(x \in \Omega\). Take any \(x, y \in \Omega\) and $0<\lambda<1$. Then \(\|x\|\leq c,\, \|y\| \leq c\), and hence
\begin{equation*}
\begin{aligned}
    \left\|\lambda x + (1 - \lambda)y\right\| &= c\left\|\lambda \frac{x}{c} + (1 - \lambda)\frac{y}{c}\right\| \\
    &\leq c\left(\max\left\{\frac{\|x\|}{c}, \frac{\|y\|}{c}\right\} - \frac{\sigma}{4}\lambda(1 - \lambda)\left\|\frac{x - y}{c}\right\|^2\right) \\
    &= \max\{\|x\|, \|y\|\} - \frac{\sigma}{4c}\lambda(1 - \lambda)\left\|x - y\right\|^2.
\end{aligned}
\end{equation*}
Hence, \(f\) is \(\frac{\sigma}{2c}\)-strongly quasiconvex on \(\Omega\). We conclude that \(f\) is strongly quasiconvex on  \(\Omega\). This completes the proof. 
\end{proof}

\begin{example} Let $H$ be a Hilbert space. We can check that 
\begin{equation*}
    \delta(\epsilon)=1-\sqrt{1-\frac{\epsilon^2}{4}}=\frac{\frac{\epsilon^2}{4}}{1+\sqrt{1-\frac{\epsilon^2}{4}}}\geq \frac{\epsilon^2}{8}\;\ \mbox{\rm for }\; \epsilon\in [0, 2].
\end{equation*}
Therefore, Assumption (A) is satisfied. Therefore, Theorem~\ref{StrongQuasiNorm} provides an alternative proof for the strong quasiconvexity of the norm functions on nonempty bounded convex sets in Hilbert spaces.     
\end{example}

\begin{example}
Theorem~\ref{StrongQuasiNorm} allows for a complete characterization of which \(\ell^p\) and \(L^p([0, 1])\) norms for \(1 < p < \infty\) are strongly quasiconvex on nonempty bounded convex sets. 

Indeed, in the case of \(L^p([0, 1])\), where $1<p\leq 2$, we have $\delta(\epsilon)\geq \frac{p-1}{8}\epsilon^2$; see \cite[Corollary~1]{Meir1984}. By Theorem~\ref{StrongQuasiNorm}, the norm function in \(L^p([0, 1])\) is strongly quasiconvex on nonempty bounded convex sets. 

Now, consider the case where $p>2$. It is well-known that (see, e.g., \cite[Theorem 2]{Hanner55})
\begin{equation*}
    \delta(\epsilon)=1-\left(1-\big(\frac{\epsilon}{2}\big)^p\right)^\frac{1}{p} \;\ \mbox{\rm for }\; \epsilon\in [0, 2].
\end{equation*}
In this case, we can check that
\begin{equation*}
    \lim_{\epsilon\to 0^+}\frac{\delta(\epsilon)}{\epsilon^2}=0.
\end{equation*}
Thus, Assumption (A) is not satisfied, and hence by Theorem~\ref{NC} the norm function is not strongly quasiconvex on the closed unit ball of the space.

\end{example}
\begin{corollary} Let \(X\) be a normed vector space, and let \(\| \cdot \|\) be its corresponding norm under Assumption {\rm (A)}. Then the function \(\| \cdot \|^a\), where $0<a<1$, is strongly quasiconvex on nonempty bounded convex sets.
\end{corollary}
\begin{proof}
Let \(\Omega \subset X\) be a nonempty bounded convex set. Let \(r > 0\) such that \(\Omega \subset B(0; r)\), \(x, y \in \Omega\) such that \(x \neq y\) and \(\|x\| \leq \|y\|\),  and \(\lambda \in (0, 1)\). Our assumptions imply that
\begin{equation*}
    \|\lambda x + (1 - \lambda)y\| < \|y\|.
\end{equation*}
If \(\|\lambda x + (1 - \lambda)y\| > 0\), then by the mean value theorem there exists \(c \in (0, r)\) such that
\begin{equation*}
\begin{aligned}
    \|y\|^a - \|\lambda x + (1 - \lambda)y\|^a &= ac^{a - 1}(\|y\| - \|\lambda x + (1 - \lambda)y\|) \\
    &\geq ar^{a - 1}(\|y\| - \|\lambda x + (1 - \lambda)y\|).
\end{aligned}
\end{equation*}
Furthermore,  \(\| \cdot \|\) is strongly quasiconvex on \(\Omega\) by theorem~\ref{StrongQuasiNorm}. Take \(\sigma > 0\) such that
\begin{equation*}
    \|y\| - \|\lambda x + (1 - \lambda)y\| \geq \frac{\sigma}{2} \lambda(1 - \lambda)\|x - y\|^2.
\end{equation*}
Therefore,
\begin{equation*}
    \|\lambda x + (1 - \lambda)y\|^a \leq \max\{\|x\|^a, \|y\|^a\} -a r^{a - 1}\frac{\sigma}{2} \lambda(1 - \lambda)\|x - y\|^2.
\end{equation*}
If \(\|\lambda x + (1 - \lambda)y\| = 0\), then since \(\|y\|^a \geq ar^{a - 1}\|y\|\) we may use the strong quasiconvexity of \(\| \cdot \|\) on \(\Omega\) again to get the above inequality. We conclude that \(\| \cdot \|^a\) is strongly quasiconvex on $\Omega$.
\end{proof}

\section{Supercoercive Properties of Strongly Quasiconvex Functions}

In this section, we study the supercoercive properties of strongly quasiconvex functions on normed vector spaces. 

\begin{definition} Let $X$ be a normed vector space and \(n\) be a positive integer. A function $g\colon X\to \oR$ is said to be $n$-{\em supercoercive} if 
\begin{equation*}
    \liminf_{\|x\|\to \infty}\frac{g(x)}{\|x\|^n}>0.
\end{equation*}
Let $K$ be a nonempty subset of $X$, and let $f\colon K\to \R$. We say that $f$ is $n$-supercoercive on $K$ if the function $g$ given by 
\begin{equation*}
    g(x)=\begin{cases} f(x)&\mbox{\rm if }x\in K,\\
    \infty &\mbox{\rm if }x\in X\setminus K,
    \end{cases}
\end{equation*}
is $n$-supercoercive.
\end{definition}

The next result is a partial extension of \cite[Theorem 1]{Lara22} to infinite-dimensional Banach spaces. Here, we need the additional assumption of lower semicontiniuty at a point.
\begin{theorem}
    Let \(\sigma > 0\), \(X\) be a normed vector space, \(K \subset X\) be nonempty and convex, and \(f\colon  K \rightarrow \mathbb{R}\) be \(\sigma\)-strongly quasiconvex such that \(f\) is lower-semicontinuous at some point \(x_0 \in K\). Then \(f\) is $2$-supercoercive on $K$.
    \end{theorem}
\begin{proof}
Take \(\delta > 0\) such that for all \(\|y - x_0\| < \delta\) with $y\in K$, we have
\begin{equation*}
    f(x_0) - 1 < f(y).
\end{equation*}
Fix \(y \in K\) and assume without loss of generality that \(\|y - x_0\| > \frac{\delta}{2}\) (we consider the limit as $\|y\|\to \infty$, so we only need to consider $y$ in $K$ such that $\|y|$ is sufficiently large).  Then there exists \(\lambda \in (0, 1)\) such that
\begin{equation*}
    \|\lambda y + (1 - \lambda)x_0 - x_0\| = \frac{\delta}{2}.
\end{equation*}
(Indeed, observe that
\begin{equation*}
    \lambda y + (1 - \lambda)x_0 - x_0 = \lambda (y - x_0),
\end{equation*}
so it suffices to take \(\lambda = \frac{\delta}{2\|y - x_0\|}\)). Let \(x = \lambda y + (1 - \lambda)x_0\) with $\lambda = \frac{\delta}{2\|y - x_0\|}$. Then by the strong quasiconvexity of \(f\), we have
\begin{equation*}
\begin{aligned}
    f(x_0) - 1 &< f(x) \\
    &\leq \max\{f(x_0), f(y)\} - \frac{\sigma}{2} \lambda(1 - \lambda)\|y - x_0\|^2 \\
    &= \max\{f(x_0), f(y)\} - \frac{\sigma}{2} \frac{\delta}{2\|y - x_0\|}\left(1 - \frac{\delta}{2\|y - x_0\|}\right)\|y - x_0\|^2 \\
    &= \max\{f(x_0), f(y)\} - \frac{\sigma}{2} \frac{\delta}{2}\left(\|y - x_0\| - \frac{\delta}{2}\right).
\end{aligned}
\end{equation*}
Consequently,
\begin{equation*}
    f(x_0) < \max\{f(x_0), f(y)\} + 1 + \frac{\sigma \delta^2}{8} - \sigma \frac{\delta}{4}\|y - x_0\|.
\end{equation*}
We claim that there exists \(M > 0\) such that for all \(y \in K\) satisfying \(\|y\| > M\), we have \(\max\{f(x_0), f(y)\} = f(y)\). Indeed, suppose on the contrary that this is not the case. Then there  exists \(\{y_n\} \subset K\) such that \(\lim_{n \rightarrow \infty} \|y_n\| = \infty\) and \(\max\{f(x_0), f(y)\} = f(x_0)\). Since \(\overline{\B}(x_0; \frac{\delta}{2})\) is bounded, there exists a sequence \(N \in \mathbb{N}\) such that for all \(n \geq N\), we have
\begin{equation*}
    \|y_n - x_0\| > \frac{\delta}{2}.
\end{equation*}
Thus,
\begin{equation*}
    f(x_0) < f(x_0) + 1 + \frac{\sigma \delta^2}{8} - \sigma \frac{\delta}{4}\|y_n - x_0\|.
\end{equation*}
Taking \(n \rightarrow \infty\) thus yields a contradiction. For all $y\in K$ with \(\|y\| > \max\{M, \|x_0\| + \delta\}\), we thus obtain
\begin{equation*}
    f(x_0) - 1 - \frac{\sigma \delta^2}{8} +  \sigma \frac{\delta}{4}\|y - x_0\| < f(y).
\end{equation*}
It follows that \(f\) is $1$-supercoercive.

Now enlarge \(M\) so that \(\|y\| > M\) both implies \(f(y) \geq f(x_0)\) and \(f\left(\frac{x_0 + y}{2}\right) \geq f(x)\). Then applying the strong quasiconvexity of \(f\) again gives us
\begin{equation*}
\begin{aligned}
    f(x) &\leq f\left(\frac{x_0 + y}{2}\right) \\
    &\leq \max\{f(x_0), f(y)\} - \frac{\sigma}{8}\|y - x_0\|^2 \\
    &= f(y) - \frac{\sigma}{8}\|y - x_0\|^2.
\end{aligned}
\end{equation*}
Thus,
\begin{equation*}
    f(x_0) + \frac{\sigma}{8}\|y - x_0\|^2 \leq f(y).
\end{equation*}
We conclude that \(f\) is 2-supercoercive, as desired.
\end{proof}
\section{Existence and Uniqueness of Minimum under Strong Quasiconvexity and Lower Semicontinuity}
Given a function \(f\colon  K \rightarrow {\mathbb{R}}\) and \(\alpha \in \mathbb{R}\), define the \(\alpha\)-sublevel set of \(f\) by
\begin{equation*}
    L_\alpha = \{x \in K: f(x) \leq \alpha\}.
\end{equation*}
\begin{theorem}\label{StrongQuasiLscMin}
Let \(\sigma > 0\), \(X\) be a reflexive Banach space, \(K \subset X\) be nonempty, closed, and convex, and \(f\colon K \rightarrow {\mathbb{R}}\) be lower semicontinuous and \(\sigma\)-strongly quasiconvex. Then there then exists \(\overline{x} \in K\) such that \(\operatorname{argmin}_K(f) = \{\overline{x}\}\).  Furthermore, for any \(y \in K\) we have
\begin{equation*}
    f(\overline{x}) + \frac{\sigma}{4}\|y - \overline{x}\|^2 \leq f(y).
\end{equation*}
\end{theorem}
\begin{proof}
Note that the assumptions on \(f\) guarantee that its sublevel sets are closed and convex, so that \(f\) is weakly lower semicontinuous \cite[Proposition 2.161]{MN22}. Since \(f\) is coercive, its sublevel sets are also bounded. Let $\gamma=f(x_0)$ for some $x_0\in K$.  It is clear that
\begin{equation*}
    \inf_{x \in K \cap L_\gamma} f(x) = \inf_{x \in K} f(x).
\end{equation*}
Since \(K\) is closed and convex, \(K \cap L_\gamma\) is closed, convex, and bounded. It follows that \(K \cap L_\gamma\) is weakly compact \cite[Corollary 3.22]{B10}. Since \(f\) is weakly lower semicontinuous, it takes a minimum \(\overline{x} \in K\).
Now, suppose on the contrary that there exist  \(x_1, x_2 \in \operatorname{argmin}_K(f)\) such that \(x_1 \neq x_2\). Since \(K\) is convex, \(\frac{x_1 + x_2}{2} \in K\). By the strong quasiconvexity of \(f\), we have
\begin{equation*}
    f\left(\frac{x_1 + x_2}{2}\right) \leq \max\{f(x_1), f(x_2)\} - \frac{\sigma}{8}\|x_1 - x_2\|^2 < \inf_{x \in K} f(x).
\end{equation*}
By contradiction, \(\operatorname{argmin}_K(f)\) must have at most one point. Therefore, \(\operatorname{argmin}_K(f) = \{\overline{x}\}\).
Now let \(y \in K\). Since \(f(y)\geq f(\overline{x})\) and \(f\left(\frac{y + \overline{x}}{2}\right) \geq f(\overline{x})\), the \(\sigma\)-strong quasiconvexity of \(f\) implies
\begin{equation*}
    f(\overline{x}) + \frac{\sigma}{8}\|y - \overline{x}\|^2 \leq f(y).
\end{equation*}
We now improve the above estimate. Let \(\lambda \in (0, 1)\). Applying the \(\sigma\)-strong quasiconvexity of \(f\), we have
\begin{equation*}
    f(\lambda y + (1 - \lambda)\overline{x}) \leq \max\{f(\overline{x}), f(y)\} - \frac{\sigma}{2} \lambda(1 - \lambda)\|y - \overline{x}\|^2 = f(y) - \frac{\sigma}{2} \lambda(1 - \lambda)\|y - \overline{x}\|^2.
\end{equation*}
Now let \(n \in \mathbb{N}\). We claim that
\begin{equation*}
    f\left(\lambda^n y + (1 - \lambda^n)\overline{x}\right) \leq f(y) - \left(\sum_{j = 0}^{n - 1}\lambda^{2j}\right)\frac{\sigma}{2} \lambda(1 - \lambda)\|y - \overline{x}\|^2.
\end{equation*}
Indeed, we have already verified the above inequality for \(n = 1\). Now suppose it holds for some \(n\in \N\). Then for \(n + 1\), we have
\begin{equation*}
\begin{aligned}
    f(\lambda^{n + 1} y + (1 - \lambda^{n + 1})\overline{x}) &= f(\lambda(\lambda^n y + (1 - \lambda^n)\overline{x}) + (1 - \lambda)\overline{x}) \\
    &\leq \max\{f(\lambda^n y + (1 - \lambda^n)\overline{x}), f(\overline{x})\} - \frac{\sigma}{2}\lambda(1 - \lambda)\|\lambda^n y + (1 - \lambda^n)\overline{x} - \overline{x}\|^2 \\
    &= f(\lambda^n y + (1 - \lambda^n)\overline{x}) - \frac{\sigma}{2}\lambda(1 - \lambda)\|\lambda^n y - \lambda^n \overline{x}\|^2 \\
    &\leq f(y) - \left(\sum_{j = 0}^{n - 1}\lambda^{2j}\right)\frac{\sigma}{2} \lambda(1 - \lambda)\|y - \overline{x}\|^2 - \lambda^{2n}\sigma\lambda(1 - \lambda)\|y - \overline{x}\|^2 \\
    &= f(y) - \left(\sum_{j = 0}^n\lambda^{2j}\right)\frac{\sigma}{2} \lambda(1 - \lambda)\|y - \overline{x}\|^2.
\end{aligned}
\end{equation*}
By induction we have the desired inequality. Since \(\overline{x}\) minimizes \(f\), we have
\begin{equation*}
\begin{aligned}
    f(\overline{x}) &\leq \liminf_{n \rightarrow \infty} f(\lambda^n y + (1 - \lambda^n)\overline{x}) \\
    &\leq \liminf_{n \rightarrow \infty} \left(f(y) - \left(\sum_{j = 0}^{n - 1}\lambda^j\right)\sigma \lambda(1 - \lambda)\|y - \overline{x}\|^2\right) \\
    &= f(y) - \left(\sum_{j = 0}^\infty\lambda^{2j}\right)\frac{\sigma}{2} \lambda(1 - \lambda)\|y - \overline{x}\|^2 \\
    &= f(y) - \frac{\sigma}{2}\frac{\lambda (1 - \lambda)}{1 - \lambda^2}\|y - \overline{x}\|^2 \\
    &= f(y) - \frac{\sigma}{2}\frac{\lambda}{1 + \lambda}\|y - \overline{x}\|^2.
\end{aligned}
\end{equation*}
Hence,
\begin{equation*}
    f(\overline{x}) + \frac{\sigma}{2} \sup_{\lambda \in (0, 1)}\left(\frac{\lambda}{1 + \lambda}\right)\|y - \overline{x}\|^2 \leq f(y).
\end{equation*}
Since
\begin{equation*}
    \sup_{\lambda \in (0, 1)}\left(\frac{\lambda}{1 + \lambda}\right) = \frac{1}{2},
\end{equation*}
we conclude that
\begin{equation*}
    f(\overline{x}) + \frac{\sigma}{4}\|y - \overline{x}\|^2 \leq f(y),
\end{equation*}
which completes the proof. 
\end{proof}

Let \(X\) be a Banach space, and let \(f\colon X \rightarrow {\mathbb{R}}\). Define \(\operatorname{prox}_f\colon X \tto X\) by
\begin{equation*}
    \operatorname{prox}_f(v) = \underset{x \in X}{\operatorname{argmin}}\left(f(x) + \frac{1}{2}\|x - v\|^2\right), \; \ v\in X.
\end{equation*}
If \(f(x) + \frac{1}{2}\|x - v\|^2\) does not have a minimum, then \(\operatorname{prox}_f(v) = \varnothing\).

\begin{corollary} Let \(X\) be a reflexive Banach space. If \(f\) is strongly quasiconvex and lower semicontinuous, then \(\operatorname{prox}_f(v)\) is nonempty for all \(v \in X\).
\end{corollary}
\begin{proof}
By theorem~\ref{StrongQuasiLscMin}, \(f\) is coercive. It follows that \(g(x) = f(x) + \frac{1}{2}\|x - v\|^2\) is coercive. Since \(f\) is strongly quasiconvex and \(\frac{1}{2}\|\cdot - v\|^2\) is convex, and both are lower semicontinuous, we see that both of the functions are weakly lower semicontinuous so that g is weakly lower semicontinuous. Since \(X\) is reflexive, \(g\) takes a minimum on \(X\), and so \(\operatorname{prox}_f(v)\) is nonempty.
\end{proof}

 \section{Some Counterexamples on Strong Quasiconvexity}
Recall the following properties of convex functions:
\begin{enumerate}
    \item If \(f: \mathbb{R}^n \rightarrow \mathbb{R}\) is convex, then \(f\) is continuous \cite[Corollary 2.152]{MN22}.
    \item If \(X\) is a normed vector space and  \(f\colon X \rightarrow \mathbb{R}\), then \(f\) is convex iff it is locally convex \cite[Corollary 2.4]{LY10}.
    \item If \(X\) is a vector space, \(\Omega \subset X\) is convex, and \(f_1,  f_2\colon  X \rightarrow \R\) are both convex, then \(f_1 + f_2\) is convex.
\end{enumerate}
Here we provide counterexamples to each of the above properties if convexity is replaced with strong quasiconvexity. We also provide an example showing that a strongly quasiconvex function may not have a minimum if it is not lower semicontinuous.
\subsection{Discontinuous Strongly Quasiconvex Functions}
\begin{example}\label{DisconSQC}
Let \(f\colon  \mathbb{R} \rightarrow \mathbb{R}\) be defined as follows:
\begin{equation*}
    f(x) = 
    \begin{cases}
        x^2 & \text{if \(x \neq 0\)}, \\
        -1 & \text{if \(x = 0\)}.
    \end{cases}
\end{equation*}
Note that the function given by \(g(x) = x^2\) is $2$-strongly convex. Let \(x, y \in \mathbb{R}\) such that \(x < y\) and \(\lambda \in (0, 1)\). We then have
\begin{equation*}
    g(\lambda x + (1 - \lambda)y) \leq \lambda g(x) + (1 - \lambda)g(y) - \lambda(1 - \lambda) |y - x|^2.
\end{equation*}
Since \(f \leq g\) and \(f(z) = g(z)\) for all \(z \neq 0\), we have
\begin{equation*}
\begin{aligned}
    f(\lambda x + (1 - \lambda)y) &\leq g(\lambda x + (1 - \lambda)y) \\
    &\leq \lambda g(x) + (1 - \lambda)g(y) - \lambda(1 - \lambda) |y - x|^2 \\
    &\leq \max\{g(x), g(y)\} - \lambda(1 - \lambda) |y - x|^2.
\end{aligned}
\end{equation*}
If \(x, y \neq 0\), then \(f(x) = g(x)\) and \(f(y) = g(y)\), so that
\begin{equation*}
    \max\{g(x), g(y)\} = \max\{f(x), f(y)\}.
\end{equation*}
If \(x = 0\), then \(y > 0\) so that \(g(y) > g(x)\), \(f(y) > f(x)\), and \(f(y) = g(y)\). These facts combined imply that
\begin{equation*}
    \max\{g(x), g(y)\} = g(y) = f(y) = \max\{f(x), f(y)\}.
\end{equation*}
A similar argument holds for \(y = 0\). Therefore, we have
\begin{equation*}
    \max\{g(x), g(y)\} - 2\lambda(1 - \lambda) |y - x|^2 = \max\{f(x), f(y)\} - \lambda(1 - \lambda) |y - x|^2.
\end{equation*}
Furthermore,
\begin{equation*}
    f(\lambda x + (1 - \lambda)y) \leq \max\{f(x), f(y)\} - \lambda(1 - \lambda) |y - x|^2.
\end{equation*}
Therefore, \(f\) is $2$-strongly quasiconvex and clearly discontinuous.
\end{example}
We now give an example of a strongly quasiconvex function that is discontinuous at infinitely many points.
\begin{example}
Define the function \(f\colon  \mathbb{R} \rightarrow \mathbb{R}\) by
\begin{equation*}
    f(x) = x^2 + \lceil |x| \rceil, \ \; x\in \R.
\end{equation*}
Fix \(x, y \in \mathbb{R}\) and assume without loss of generality that \(|x| \leq |y|\). Then clearly \(f(x) \leq f(y)\). Define the function \(g\colon  \mathbb{R} \rightarrow \mathbb{R}\) by
\begin{equation*}
    g(z) = z^2 + \lceil |y| \rceil, \ \; z\in \R.
\end{equation*}
Note that \(g\) is $2$-strongly quasiconvex and \(g(x) \leq g(y)\). Then for every \(\lambda \in [0, 1]\), we have
\begin{equation*}
\begin{aligned}
    g(\lambda x + (1 - \lambda)y) &\leq \max\{g(x), g(y)\} - \lambda (1 - \lambda)|x - y|^2 \\
    &= g(y) - \lambda (1 - \lambda)|x - y|^2
\end{aligned}
\end{equation*}
Now observe that \(g(y) = f(y) = \max\{f(x), f(y)\}\) and that
\begin{equation*}
    |\lambda x + (1 - \lambda)y| \leq \lambda |x| + (1 - \lambda)|y| \leq |y|.
\end{equation*}
Thus,
\begin{equation*}
    f(\lambda x + (1 - \lambda)y) \leq g(\lambda x + (1 - \lambda)y).
\end{equation*}
We conclude that
\begin{equation*}
    f(\lambda x + (1 - \lambda)y) \leq \max\{f(x), f(y)\} - \lambda (1 - \lambda)|x - y|^2
\end{equation*}
Therefore, \(f\) is $2$-strongly quasiconvex and clearly discontinuous at every integer.
\end{example}
\subsection{Locally Strong Quasiconvexity and Strong Quasiconvexity}
We first define local \(\sigma\)-quasiconvexity analogously to local convexity.
\begin{definition}
Let \(X\) be a normed vector space, let \(f\colon X \rightarrow \R\), and let \(\sigma \geq 0\). We say that \(f\) is {\em locally} $\sigma$-{\em quasiconvex} if for all \(x \in X\), there exists \(\delta > 0\) such that \(f\) is $\sigma$-quasiconvex on \(\B(x; \delta)\).
\end{definition}
\begin{example}
Let \(f(x) = x\). Obviously, \(f\) is convex. Let \(x, y \in \mathbb{R}\) such that \(x < y\) and \(\lambda \in (0, 1)\). It is clear that
\begin{equation*}
    f(y) > f(x).
\end{equation*}
Hence, \(f(y) = \max\{f(x), f(y)\}\). For \(f(y) - f(\lambda x + (1 - \lambda)y)\), we have:
\begin{equation*}
    f(y) - f(\lambda x + (1 - \lambda)y) = \lambda (y - x)
\end{equation*}
Suppose that \(y - x < 1\). Then
\begin{equation*}
    \lambda (y - x) \geq \lambda(1 - \lambda)(y - x)^2.
\end{equation*}
Hence
\begin{equation*}
    f(y) - f(\lambda x + (1 - \lambda)y) \geq \lambda(1 - \lambda)(y - x)^2.
\end{equation*}
Then we have
\begin{equation*}
    f(\lambda x + (1 - \lambda)y) \leq \max\{f(x), f(y)\} - \lambda(1 - \lambda)|y - x|^2.
\end{equation*}
It follows that \(f\) is $2$-strongly quasiconvex on any open interval of length $1$. Thus, \(f\) is locally $2$-strongly quasiconvex.

Now, let \(r > 0\), \(x = -r\), and \(y = r\). Take \(\lambda = \frac{1}{2}\) so that \(\lambda x + (1 - \lambda)y = 0\). Clearly,
\begin{equation*}
    f(\lambda x + (1 - \lambda)y) = 0.
\end{equation*}
For \(f(y) - \lambda(1 - \lambda)|y - x|^2\), we have
\begin{equation*}
    f(y) - \lambda(1 - \lambda)|y - x|^2 = r - r^2.
\end{equation*}
Thus, if \(r > 1\), \(f\) fails the $2$-strong quasiconvexity condition. In fact, by taking \(r\) arbitrarily large, we see that \(f\) fails to be \(\sigma\)-strongly quasiconvex for any \(\sigma > 0\).
\end{example}
\subsection{Sum of Strongly Quasiconvex Functions}
\begin{example}
Let \(f\colon \mathbb{R} \rightarrow \mathbb{R}\) be the following function:
\begin{equation*}
    f(x) = 
    \begin{cases}
        x^2 & \text{if \(x \neq 0\)},\\
        -1 & \text{if \(x = 0\)}.
    \end{cases}
\end{equation*}
By example~\ref{DisconSQC}, \(f\) is $2$-strongly quasiconvex. Now, define \(g\colon  \mathbb{R} \rightarrow \mathbb{R}\) by
\begin{equation*}
    g(x) = f(x + 1) + f(x - 1).
\end{equation*}
Note that \(f(\cdot - 1)\) and \(f(\cdot + 1)\) are also $2$-strongly quasiconvex by Corollary~\ref{a comp}. Observe that
\begin{equation*}
    g(x) =
    \begin{cases}
        2(x^2 + 1) & \text{if \(|x| \neq 1\)}, \\
        3 & \text{if \(|x| = 1\)}.
    \end{cases}
\end{equation*}
Let \(x = -1\), \(y = 1\) and take \(\lambda \in (0, 1)\) such that
\begin{equation*}
    \lambda x + (1 - \lambda)y = \sqrt{\frac{3}{4}}.
\end{equation*}
Then
\begin{equation*}
    g(\lambda x + (1 - \lambda)y) = \frac{7}{2} > 3 = \max\{g(x), g(y)\}.
\end{equation*}
Thus, \(g\) is not even quasiconvex, so it is not strongly quasiconvex.
\end{example}

\subsection{A Strongly Quasiconvex Function that is not Lower Semicontinuous and takes no Minimum}
\begin{example}
Let \(f\colon [0,\infty) \rightarrow \mathbb{R}\) be the following function:
\begin{equation*}
    f(x) =
    \begin{cases}
        1 & \text{if \(x = 0\)}, \\
        x^2 & \text{if \(x > 0\)}.
    \end{cases}
\end{equation*}
Let \(x, y \in [0, \infty)\) and \(\lambda \in (0, 1)\). If \(x,y > 0\). Then
\begin{equation*}
    f(\lambda x + (1 - \lambda)y) \leq \max\{f(x), f(y)\} - \lambda(1 - \lambda)\|y - x\|^2.
\end{equation*}
Now, suppose that \(x = 0 < y\). Since \(\lambda x + (1 - \lambda)y > 0\), we have
\begin{equation*}
\begin{aligned}
    f(\lambda x + (1 - \lambda)y) &\leq \max\{0, y^2\} - \lambda(1 - \lambda)\|y - x\|^2 \\
    &\leq \max\{1, y^2\} - \lambda(1 - \lambda)\|y - x\|^2 \\
    &= \max\{f(x), f(y)\} - \lambda(1 - \lambda)\|y - x\|^2.
\end{aligned}
\end{equation*}
Thus, \(f\) is strongly quasiconvex, but clearly is lower semicontinuous and takes no minimum.
\end{example}
{\bf Acknowledgements}. We would like to express our gratitude to Prof. Nguyen Dong Yen for the valuable discussions on the concept of strong quasiconvexity and for his insightful question, addressed in Example~\ref{DisconSQC}, during the XIV International Symposium on Generalized Convexity and Monotonicity in Pisa, Italy.

	\end{document}